\documentclass[11pt]{amsart}

\title[The 2D Kuramoto-Sivashinsky Equation]{Global solutions of the two-dimensional Kuramoto-Sivashinsky equation with a linearly growing mode in each direction}

\author[David M. Ambrose]{David M. Ambrose$^{*}$}
\thanks{$^{*}$corresponding author}
\address{Department of Mathematics, Drexel University, Philadelphia, PA 19104, USA}
\email{dma68@drexel.edu}

\author[Anna L. Mazzucato]{Anna L. Mazzucato}
\address{Department of Mathematics, Penn State University, University Park, PA 16802, USA}
\email{alm24@psu.edu}

\newtheorem{theorem}{Theorem}
\newtheorem{proposition}{Proposition}
\newtheorem{remark}{Remark}

\begin{document}

\begin{abstract}
We consider the Kuramoto-Sivashinsky equation in two space dimension. We establish the first proof of global
existence of solutions in the presence of a linearly growing mode in both
spatial directions for sufficiently small data.  We develop a new method to this end, categorizing wavenumbers as low (linearly growing modes),
intermediate (linearly decaying modes that serve as energy sinks for the low modes), and high (strongly linearly
decaying modes).  The low and intermediate modes are controlled by means of a Lyapunov function, while the high
modes are controlled with operator estimates in function spaces based on the Wiener algebra.
\end{abstract}

\maketitle

\section{Introduction}

We study the Kuramoto-Sivashinsky equation on a rectangular domain  $[0,L_{1}]\times[0,L_{2}]$ under periodic boundary conditions:
\begin{equation}\label{KS}
\psi_{t}=-\Delta^{2}\psi-\Delta\psi-|\nabla\psi|^{2}.
\end{equation}
The Kuramoto-Sivashinsky equation is a well-known model of flame front propagation and was first
derived in \cite{kuramoto}, \cite{sivashinsky}.
We will prove a global existence theorem for solutions with sufficiently small data in a suitable function space in the presence of one linearly growing mode in each direction.  There are a number
of global existence theorems for the Kuramoto-Sivashinsky equation in one spatial dimension 
\cite{goodman}, \cite{NST}, \cite{tadmor}, and detailed studies of the asymptotics of these solutions
\cite{bronski}, \cite{otto1}, \cite{otto2}, \cite{otto3} (see also \cite{KTZ18} for the effect of adding dispersion).  These one-dimensional results rely on a particular structure of the nonlinearity that is not present in two spatial dimensions, and thus there are far fewer global results available in the  two-dimensional case.

In the spatially periodic case, the dynamics of the Kuramoto-Sivashinsky equation are in part governed by the size of
the domain, as this determines how many linearly growing Fourier modes are present.
In two spatial dimensions, most
global existence results in the literature are inherently anisotropic, that is, the length of one period is small compared to that of the other period and/or the size of the initial data.
In thin domains, solutions are shown to remain close to one-dimensional solutions.  Such
studies were initiated by Sell and Taboada \cite{sell}.  Other anisotropic global existence theorems are
the works \cite{igor}, \cite{kukavicaMassatt}, \cite{molinetThin}.  
Then there are global existence and singularity formation results for modified equations.
The fourth-order nature of the parabolic evolution \eqref{KS}  implies the absence
of a maximum principle for the linearized evolution.  Some authors have shown that related systems with maximum principles
do have global solutions \cite{larios}, \cite{molinet-burgers}, while others have  modified the nonlinear term,
showing that related equations have finite-time singularities \cite{titi} or global solutions \cite{campos}, 
\cite{pinto}; see also \cite{papageorgiou2} for a numerical study of a modified equation. 
The second author and Feng have shown that a modification of \eqref{KS} with additional advection also
has global solutions \cite{mazzucatoFeng}.
Rather than modifying the equation or relying on anisotropy, the authors have previously given a global
existence theorem for the two-dimensional Kuramoto-Sivashinsky equation, but under the requirement
that the domain size be sufficiently small, a requirement that precludes growing modes \cite{ambroseMazzucato}.  Nonlinear stability of the zero solution  and decay rates in the long time limit  for a generalized Kuramoto-Sivashinsky equation with damping were obtained in \cite{ZhaoTang00} under conditions on the coefficients that ensure stability for the linearized operator  and for small data in $L^2$.

The linear operators $-\Delta^{2}$ and $-\Delta$ on the right-hand side of \eqref{KS} may be viewed as being
in competition with each other; these represent a higher-order forward parabolic effect and a lower-order
backward parabolic effect, which gives rise to large-scale instabilities.  Since $-\Delta^2-\Delta$ is an elliptic operator, there are at most finitely many linearly growing Fourier modes forward in time.
To be precise, if $L_{1}$ and $L_{2}$ are each in the interval $(0,2\pi),$
then there are no linearly growing Fourier modes in \eqref{KS}, and this is the case studied in 
\cite{ambroseMazzucato}.  In the current study,  
by taking each of $L_{1}$ and $L_{2}$ slightly larger than $2\pi,$  we ensure that there is 
exactly one linearly growing
mode in each of the $x$-direction and the $y$-direction. 
In all previous global existence results for the two-dimensional Kuramoto-Sivashinsky equation,
either there were no linearly growing modes at all \cite{ambroseMazzucato}, 
or (in the strongly anisotropic works \cite{igor}, \cite{molinetThin}, \cite{sell}) 
the linearly growing modes were only in one direction.
The current work is therefore the first global existence theorem for the two-dimensional Kuramoto-Sivashinsky 
equation to allow a growing mode in each spatial direction.
We note that the 
interested reader might see \cite{papageorgiou} for a detailed numerical study of the dependence of the 
dynamics of solutions on the size of the spatial domain/the number of linearly growing Fourier modes present.
Numerically, in one space dimension one observes that the $L^2$-norm of $\nabla\psi$ remains bounded in time even when growing modes are present. Hence, the non-linearity has a restoring effect on the large-scale unstable modes. This mechanism was rigorously investigated in \cite{NST85} in two cases: in one dimension for even solutions; in dimension 2 and 3, under the assumption of a global bound on the $H^1$ norm of the solution, which implies global existence. The authors of
\cite{NST85} rely on Lyapunov function techniques, which are also at the core of our proof,  and are able to estimate the number of determining modes and the size of the attractor in terms of the assumed $H^1$ bound and in terms of the period.

The method of proof of our main theorem
primarily combines ideas from prior work of the authors \cite{ambroseMazzucato}
and from the one-dimensional global existence theorem of Goodman \cite{goodman}.  
We will now describe
the formulation of the problem to be used and how these ideas come into play.
  
We immediately notice that, while the mean of $\psi$, $\bar{\psi},$ is not preserved under the time evolution, its growth is governed by the $L^{2}$-norm of the gradient of $\psi$, which does not depend on the mean itself.
As a matter of fact, if we define $\phi=\mathbb{P}_{0}\psi,$
where $\mathbb{P}_{0}$ is the projection which removes the mean of a periodic function,  the equation
satisfied by $\phi$ is
\begin{equation}\label{eq:phiEq}
\phi_{t}=-\Delta^{2}\phi-\Delta\phi-\mathbb{P}_{0}|\nabla\phi|^{2}.
\end{equation}
The evolution equation for $\bar{\psi}$ is then
\begin{equation}\label{meanEquation}
\bar{\psi}_{t}=-\frac{1}{L_{1}L_{2}}\int_{0}^{L_{1}}\int_{0}^{L_{2}}|\nabla\phi|^{2}\ dxdy.
\end{equation}
We therefore see that the mean of $\psi$ exists and is finite at time $T$ as long as $\phi\in L^{1}([0,T];\dot{H}^{1})$, where $\dot{H}^1$ denotes the homogeneous $L^2$-Sobolev space of order 1. We  concentrate on solving \eqref{eq:phiEq} from now on.

As in Nicolaenko, Scheurer, and Temam \cite{NST}, we consider symmetric solutions:
\begin{equation}\label{evenSymmetry}
\phi(x,y,t)=\sum_{k,j\geq 1}a_{k,j}(t)\cos\left(\frac{2\pi kx}{L_{1}}\right)\cos\left(\frac{2\pi jy}{L_{2}}\right).
\end{equation}
We introduce a decomposition of the Fourier modes into three categories.  With our choice that $L_{1}$ and
$L_{2}$ are each slightly larger than $2\pi,$ we have exactly two linearly growing Fourier modes, $a_{1,0}$ and $a_{0,1};$
these linearly growing modes are the first type that we treat specially.  We next take two intermediate modes,
which are the $a_{2,0}$ and $a_{0,2}$ modes; these are linearly decaying modes that we use to absorb
energy from the lowest modes.  Finally, our third category consists of all remaining Fourier modes; we consider
these to be strongly decaying.

We let $\mathbb{P}_{5}$ be the projection onto the complement of the span of the  4 modes introduced above.  We may then write $\phi$ as
\begin{multline}\label{phiDefn}
\phi(x,y,t)=a_{1,0}(t)\cos\left(\frac{2\pi x}{L_{1}}\right)+a_{2,0}(t)\cos\left(\frac{4\pi x}{L_{1}}\right)
\\
+a_{0,1}(t)\cos\left(\frac{2\pi y}{L_{2}}\right)+a_{0,2}(t)\cos\left(\frac{4\pi y}{L_{2}}\right)
+w(x,y,t),
\end{multline}
where $w=\mathbb{P}_{5}\phi.$  The KSE is equivalent, at least formally, to a coupled system of 5 equations, 4 ODEs for the modes
$a_{1,0},$ $a_{2,0},$ $a_{0,1},$ $a_{0,2},$ and a PDE for $w.$  Our first goal is to derive this coupled system. Throughout, for ease of notation, we will denote derivatives as subscripts, so $a_{1,0 t}=\frac{d}{dt} a_{1,0}$.

We introduce some notation for the coefficients of the linear terms in the modes that we treat specially:
\begin{align}\label{eq:epBdef}
\varepsilon_{i}&=-\left(\frac{2\pi}{L_{i}}\right)^{4}+\left(\frac{2\pi}{L_{i}}\right)^{2},\\
B_{i}&=\left(\frac{4\pi}{L_{i}}\right)^{4}-\left(\frac{4\pi}{L_{i}}\right)^{2}, \nonumber
\end{align}
where $i\in\{1,2\}.$
We will sometimes denote $\varepsilon=\max\{\varepsilon_{1},\varepsilon_{2}\}$ as well.
Of course, since $2\pi<L_{i}<4\pi$,  we have that all of these coefficients are positive.
Furthermore, by taking $L_i$ only slightly larger than $2\pi$ we can make $\varepsilon_i$ arbitrarily small.
Using this notation, our equations for the first two modes in the $x$-direction become
\begin{equation}\nonumber
a_{1,0t}=\varepsilon_{1}a_{1,0}+\frac{8\pi^{2}}{L_{1}^{2}}a_{1,0}a_{2,0}+F_{1,0,x}+F_{1,0,y},
\end{equation}
\begin{equation}\nonumber
a_{2,0t}=-B_{1}a_{2,0}-\frac{2\pi^{2}}{L_{1}^{2}}a_{1,0}^{2}+F_{2,0,x}+F_{2,0,y}.
\end{equation}
Here, we have brought out the quadratic interactions between these two modes and we consider the rest
of the nonlinearity to be a smaller remainder.  We will give formulas for the forcing functions $F_{i,0,x}$ and $F_{i,0,y}$, $i=1,2$, in Section \ref{fSection} below.
Similarly, the equations for the first two modes in the $y$-direction are
\begin{equation}\nonumber
a_{0,1t}=\varepsilon_{2}a_{0,1}+\frac{8\pi^{2}}{L_{2}^{2}}a_{0,1}a_{0,2}+F_{0,1,x}+F_{0,1,y},
\end{equation}
\begin{equation}\nonumber
a_{0,2t}=-B_{2}a_{0,2}-\frac{2\pi^{2}}{L_{2}^{2}}a_{0,1}^{2}+F_{0,2,x}+F_{0,2,y}.
\end{equation}
Finally, we may write the evolution equation for $w$ simply as
\begin{equation}\label{wEqn}
w_{t}=-\Delta^{2}w-\Delta w+\mathbb{P}_{5}\left((\partial_{x}\phi)^{2}+(\partial_{y}\phi)^{2}\right),
\end{equation}
where $\phi$ and $w$ are related through \eqref{phiDefn}.

As we have indicated already, the four special modes 
will be treated with a Lyapunov function, generalizing Goodman's result for a toy model \cite{goodman}.
In Goodman's case, energy was
conserved by the nonlinear terms. The conservation of energy (i.e., a conserved $L^{2}$ norm) does not hold
in dimension greater than one, due to the form of the non-linearity.  We observe, however, that Goodman's argument is more robust than
this, and can be modified to handle the presence of small forcing.
This Lyapunov function argument will show that the first $4$ modes remains of size $\varepsilon^{1/2}$, if initially of that size. This result is inherently a non-linear effect, since the first two modes are linearly growing in fact.
For the 2-modes  $a_{2,0}$ and $a_{0,2}$, this bound can be
improved.  We will find that the size of each of $a_{2,0}$ and $a_{0,2}$ is then at most proportional to $\varepsilon$, if initially of that size.
The norm of $w$ (in a function space related to the Wiener
algebra) will be shown to be bounded by $\varepsilon^{3/2}$, if initially of that size.  The method of employing functions spaces based on the Wiener algebra was used
previously by the authors in \cite{ambroseMazzucato}, and is inspired by the work of Duchon and Robert
on vortex sheets \cite{duchonRobert} (cf. also \cite{BJMT14}).  The first author and his collaborators have additionally developed and used
the technique in \cite{ambroseMFG1}, \cite{ambroseMFG2}, \cite{ambroseEpitaxial}, \cite{ambroseBonaMilgrom},
\cite{milgromAmbrose}.

The following is the (non-technical version of) our main theorem:
\begin{theorem}\label{nontechnicalTheorem}
There exists $\varepsilon_{*}>0$ such that for any $\varepsilon\in(0,\varepsilon_{*}),$ if
\begin{equation}\nonumber
a_{1,0}(0)\sim\varepsilon^{1/2},\qquad a_{2,0}(0)\sim\varepsilon,
\end{equation}
\begin{equation}\nonumber
a_{0,1}(0)\sim\varepsilon^{1/2},\qquad a_{0,2}(0)\sim\varepsilon,\qquad w_{0}\sim \varepsilon^{3/2},
\end{equation}
then the 2D Kuramoto-Sivashinsky equation with these data has a solution on an arbitrary time interval $[0,T]$.
\end{theorem}

We will give more precise bounds on the initial data and will state a technical version of 
the theorem later in Theorem \ref{mainTheoremTechnical} in Section \ref{inductionSection}. 

While we have not carried out the proof of our main theorem in the absence of the even symmetry reflected
in \eqref{evenSymmetry}, we expect that this symmetry is not critical for achieving the result.  Our proof relies on two main 
ingredients, neither of which require this symmetry; the proof, however, would certainly be more complicated in the general
case.  One of our main ingredients is operator estimates in function spaces related to the Wiener algebra, and for these
estimates, the symmetry is completely immaterial.  The other main ingredient is Goodman's Lyapunov function argument.
Goodman actually introduced two Lyapunov functions, one for a toy model and one for the full 
one-dimensional Kuramoto-Sivashinsky
equation in the absence of symmetry.  
In the present work we generalize the Goodman result on the toy model, so using the simpler of the two
Lyapunov functions.  In the general case we expect that using a different Lyapunov function inspired by Goodman's 
argument would provide the same result as we prove here.

We focus on the two-dimensional case, 
which is the most physically motivated case as compared to higher dimensions.
Indeed, \eqref{KS} may be obtained from a coordinate-free 
model for the evolution of a flame front, which is modeled as a parametric surface \cite{frankelSivashinsky1987},
\cite{frankelSivashinsky1988}. 
It would nevertheless be very interesting 
to investigate whether our 
global existence result extends to higher dimensions. The Duchon-Robert argument to control the remainder is based on the 
Wiener algebra and hence it does not rely on dimension-dependent embeddings. The key step in such an extension is the 
analysis of the reduced system for the first modes in each direction, which should still only contain quadratic interactions up 
the the remainders.  We do expect the argument to carry over to higher dimensions.

The plan for the rest of the paper is as follows.  In Section \ref{fSection}, we complete the description of the evolution equations satisfied by the components in \eqref{phiDefn} by detailing formulas for the forcing functions.  In Section \ref{iterationSection},
we set up an iterative scheme, and we prepare to make estimates, which will be uniform in the iteration parameter.
We develop propositions that give these uniform estimates on $a_{1,0},$ $a_{2,0},$ $a_{0,1},$ and $a_{0,2}$
in Section \ref{goodmanSection}.  We then develop tools that will give uniform bounds on $w$ in Section
\ref{duchonRobertSection}.  The uniform bounds are established, and the limit of the iterates is taken, in
Section \ref{inductionSection}.  We then make some concluding remarks on future directions in Section 
\ref{conclusion}.

\subsection*{Acknowledgments}
The first author is grateful to the National Science Foundation for support through grant DMS-1907684.
The second author is grateful to the National Science Foundation for support through grant DMS-1909103.

\subsection{Formulas for the forcing functions}\label{fSection}

We let $\mathbb{P}_{1,0}$ be the projection onto the $(1,0)$ Fourier mode and we let $\mathbb{P}_{2,0}$ be the 
projection onto the $(2,0)$ Fourier mode.
Similarly, we let $\mathbb{P}_{0,1}$ be the projection onto the $(0,1)$ Fourier mode and we let $\mathbb{P}_{0,2}$ be the projection onto the $(0,2)$ Fourier mode.

We will determine  $F_{1,0,x}$ and $F_{1,0,y}$ by projecting the nonlinear term onto the $(1,0)$ mode and then separating out certain quadratic interactions:
\begin{align}\label{startFCalculation}
 &\mathbb{P}_{1,0}\left(\left(\phi_{x}\right)^{2}\right)=\left[-\frac{8\pi^{2}}{L_{1}^{2}}a_{1,0}a_{2,0}+F_{1,0,x}\right]
\cos\left(\frac{2\pi x}{L_{1}}\right),\\
&\mathbb{P}_{1,0}\left(\left(\phi_{y}\right)^{2}\right)=F_{1,0,y}\cos\left(\frac{2\pi x}{L_{1}}\right). \nonumber
\end{align}

We decompose $(\phi_{x})^{2}$ as follows:
\begin{equation}\nonumber
(\phi_{x})^{2}=\sum_{i=1}^{6}\Psi_{i}^{x},
\end{equation}
with the terms $\Psi_{i}^{x}$ defined as
\begin{equation}\nonumber
\Psi_{1}^{x}=\frac{4\pi^{2}(a_{1,0})^{2}}{L_{1}^{2}}\sin^{2}\left(\frac{2\pi x}{L_{1}}\right),
\end{equation}
\begin{equation}\nonumber
\Psi_{2}^{x}=\frac{16\pi^{2}(a_{2,0})^{2}}{L_{1}^{2}}\sin^{2}\left(\frac{4\pi x}{L_{1}}\right),
\end{equation}
\begin{equation}\nonumber
\Psi_{3}^{x}=(w_{x})^{2},
\end{equation}
\begin{equation}\nonumber
\Psi_{4}^{x}=\frac{16\pi^{2}a_{1,0}a_{2,0}}{L_{1}^{2}}\sin\left(\frac{2\pi x}{L_{1}}\right)\sin\left(\frac{4\pi x}{L_{1}}\right),
\end{equation}
\begin{equation}\nonumber
\Psi_{5}^{x}=-\frac{4\pi a_{1,0}}{L_{1}}w_{x}\sin\left(\frac{2\pi x}{L_{1}}\right),
\end{equation}
\begin{equation}\nonumber
\Psi_{6}^{x}=-\frac{8\pi a_{2,0}}{L_{1}}w_{x}\sin\left(\frac{4\pi x}{L_{1}}\right).
\end{equation}

The following equation then defines $F_{1,0,x},$,
after making elementary calculations using trigonometric identities: 
\begin{equation}\nonumber
F_{1,0,x}\cos\left(\frac{2\pi x}{L_{1}}\right)
=\mathbb{P}_{1,0}\left[\Psi_{3}^{x}+\Psi_{5}^{x}+\Psi_{6}^{x}\right].
\end{equation}

To compute $F_{1,0,y},$ we need the corresponding decomposition of $\phi_{y}^{2}:$
\begin{equation}\nonumber
\phi_{y}^{2}=\sum_{i=1}^{6}\Psi_{i}^{y},
\end{equation}
with the terms $\Psi_{i}^{y}$ defined as
\begin{equation}\nonumber
\Psi_{1}^{y}=\frac{4\pi^{2}}{L_{2}^{2}}a_{0,1}^{2}\sin^{2}\left(\frac{2\pi y}{L_{2}}\right),
\end{equation}
\begin{equation}\nonumber
\Psi_{2}^{y}=\frac{16\pi^{2}}{L_{2}^{2}}a_{0,2}^{2}\sin^{2}\left(\frac{4\pi y}{L_{2}}\right),
\end{equation}
\begin{equation}\nonumber
\Psi_{3}^{y}=w_{y}^{2},
\end{equation}
\begin{equation}\nonumber
\Psi_{4}^{y}=\frac{16\pi^{2}}{L_{2}^{2}}a_{0,1}a_{0,2}\sin\left(\frac{2\pi y}{L_{2}}\right)\sin\left(\frac{4\pi y}{L_{2}}\right),
\end{equation}
\begin{equation}\nonumber
\Psi_{5}^{y}=-\frac{4\pi a_{0,1}}{L_{2}}w_{y}\sin\left(\frac{2\pi y}{L_{2}}\right),
\end{equation}
\begin{equation}\nonumber
\Psi_{6}^{y}=-\frac{8\pi a_{0,2}}{L_{2}}w_{y}\sin\left(\frac{4\pi y}{L_{2}}\right).
\end{equation}

Again using some trigonometric identities and making other calculations, the equation for $F_{1,0,y}$ is then
\begin{equation}\nonumber
F_{1,0,y}\cos\left(\frac{2\pi x}{L_{1}}\right)=\mathbb{P}_{1,0}[\Psi_{3}^{y}+\Psi_{5}^{y}+\Psi_{6}^{y}].
\end{equation}

We now calculate the remaining forcing functions in a similar way, in particular only the third, fifth, and sixth components of $\Psi$ enter
into the formulas:
\begin{equation}\nonumber
F_{2,0,x}\cos\left(\frac{4\pi x}{L_{1}}\right)=\mathbb{P}_{2,0}\left[\Psi_{3}^{x}+\Psi_{5}^{x}+\Psi_{6}^{x}\right].
\end{equation}
\begin{equation}\nonumber
F_{2,0,y}\cos\left(\frac{4\pi x}{L_{1}}\right)=\mathbb{P}_{2,0}[\Psi_{3}^{y}+\Psi_{5}^{y}+\Psi_{6}^{y}].
\end{equation}
\begin{equation}\nonumber
F_{0,1,x}\cos\left(\frac{2\pi y}{L_{2}}\right)=\mathbb{P}_{0,1}[\Psi_{3}^{x}+\Psi_{5}^{x}+\Psi_{6}^{x}],
\end{equation}
\begin{equation}\nonumber
F_{0,1,y}\cos\left(\frac{2\pi y}{L_{2}}\right)=\mathbb{P}_{0,1}[\Psi_{3}^{y}+\Psi_{5}^{y}+\Psi_{6}^{y}],
\end{equation}
\begin{equation}\nonumber
F_{0,2,x}\cos\left(\frac{4\pi y}{L_{2}}\right)=\mathbb{P}_{0,2}[\Psi_{3}^{x}+\Psi_{5}^{x}+\Psi_{6}^{x}],
\end{equation}
\begin{equation}\nonumber
F_{0,2,y}\cos\left(\frac{4\pi y}{L_{2}}\right)=\mathbb{P}_{0,2}[\Psi_{3}^{y}+\Psi_{5}^{y}+\Psi_{6}^{y}].
\end{equation}

\section{Iterative scheme}\label{iterationSection}

We will solve the coupled system of ODEs for the $4$ specialized modes and the PDE for the remainder $w$ via an iterative scheme for $\phi^n$, where $\phi^n$ is defined by:
\begin{multline}\label{iteratedPhi}
\phi^{n}(x,y,t)=a_{1,0}^{n}(t)\cos\left(\frac{2\pi x}{L_{1}}\right)+a_{2,0}^{n}(t)\cos\left(\frac{4\pi x}{L_{1}}\right)
\\
+a_{0,1}^{n}(t)\cos\left(\frac{2\pi y}{L_{2}}\right)+a_{0,2}^{n}(t)\cos\left(\frac{4\pi y}{L_{2}}\right)
+w^{n}(x,y,t).
\end{multline}
In the scheme, the forcing terms are given by formulas corresponding to those in Section \ref{fSection} in a straightforward way.

We start by giving the equations for the $a^{n+1}$ coefficients:
\begin{equation}\nonumber
a_{1,0t}^{n+1}=\varepsilon_{1}a_{1,0}^{n+1}+\frac{8\pi^{2}}{L_{1}^{2}}a_{1,0}^{n+1}a_{2,0}^{n+1}
+F_{1,0,x}^{n}+F_{1,0,y}^{n},
\end{equation}
\begin{equation}\nonumber
a_{2,0t}^{n+1}=-B_{1}a_{2,0}^{n+1}-\frac{2\pi^{2}}{L_{1}^{2}}(a_{1,0}^{n+1})^{2}+F_{2,0,x}^{n}+F_{2,0,y}^{n},
\end{equation}
\begin{equation}\nonumber
a_{0,1t}^{n+1}=\varepsilon_{2}a_{0,1}^{n+1}+\frac{8\pi^{2}}{L_{2}^{2}}a_{0,1}^{n+1}a_{0,2}^{n+1}
+F_{0,1,x}^{n}+F_{0,1,y}^{n},
\end{equation}
\begin{equation}\nonumber
a_{0,2t}^{n+1}=-B_{2}a_{0,2}^{n+1}-\frac{2\pi^{2}}{L_{2}^{2}}(a_{0,1}^{n+1})^{2}+F_{0,2,x}^{n}+F_{0,2,y}^{n},
\end{equation}
To complete the scheme, we also give the iterated version of \eqref{wEqn} for $w^n$:
\begin{equation}\label{iteratedW}
w^{n+1}_{t}=-\Delta^{2}w^{n+1}-\Delta w^{n+1}+\mathbb{P}_{5}\left((\partial_{x}\phi^{n})^{2}
+(\partial_{y}\phi^{n})^{2}\right).
\end{equation}

The iterated system is taken with initial data that do not depend on $n,$ namely,
\begin{equation}\nonumber
a_{1,0}^{n+1}(t)=a_{1,0}(0),
\qquad
a_{2,0}^{n+1}(t)=a_{2,0}(0),
\end{equation}
\begin{equation}\nonumber
a_{0,1}^{n+1}(t)=a_{0,1}(0),
\qquad
a_{0,2}^{n+1}(t)=a_{0,2}(0),
\qquad
w^{n+1}=w_{0}.
\end{equation}

\subsection{List of constants}
For convenience, we label some combinations of constants that will appear in ensuing calculations.  We first introduce $M_{1,1}$
and $M_{1,2},$ which will be used in the
bounds for $a_{1,0}^{n}$ and $a_{0,1}^{n}:$
\begin{equation}\nonumber
M_{1,1}=\frac{12B_{1}L_{1}^{4}}{\pi^{4}},\qquad M_{1,2}=\frac{12B_{2}L_{2}^{4}}{\pi^{4}}.
\end{equation}
The following constants will be used in the bounds for $a_{2,0}^{n}$ and $a_{0,2}^{n}:$
\begin{equation}\nonumber
M_{2,1}=\frac{8\pi^{2}M_{1,1}}{L_{1}^{2}}.\qquad M_{2,2}=\frac{8\pi^{2}M_{1,2}}{L_{2}^{2}}.
\end{equation}
The constant $M_3$ will be used in the  bound for $w^{n}:$
\begin{equation}\nonumber
M_{3}=\max\left\{6K_{1}\left(2M_{1,1}^{1/2}M_{2,1}K_{2}\right), 6K_{1}\left(2M_{1,2}^{1/2}M_{2,2}K_{2}\right)\right\}.
\end{equation}

The formula above for $M_{3}$ involves two other constants, $K_{1}$ and $K_{2}.$  Of these, $K_{1}$ is a bound
for the operator norm of an integral term in the mild formulation of the equation for $w^{n}$; this formulation will be developed in Section \ref{duchonRobertSection} below.
To specify the constant $K_{1}$ we need to specify a set, $A,$ of special wavenumber pairs:
\begin{equation}\nonumber
A=\{(0,0), (1,0), (2,0), (0,1), (0,2)\}.
\end{equation}
Then $K_{1}$ is given by 
\begin{equation}\label{definitionOfK1}
K_{1}=\sup_{(k,j)\in\mathbb{Z}^{2}\setminus A}\frac{1+|k|+|j|}{-\sigma(k,j)},
\end{equation}
where $\sigma$ is the symbol of the linearized KSE operator $-\Delta^2-\Delta$,
\begin{equation}\label{eq:sigmaDef}
\sigma(k,j)=-\left(\left(\frac{2\pi k}{L_{1}}\right)^{2}+\left(\frac{2\pi j}{L_{2}}\right)^{2}\right)^{2}
+\left(\frac{2\pi k}{L_{1}}\right)^{2}+\left(\frac{2\pi j}{L_{2}}\right)^{2}.
\end{equation}
We notice that the denominator in \eqref{definitionOfK1}
is quartic with respect to $k$ and $j$, while the numerator is linear.  Also, the denominator
is always positive, as the only pairs for which the denominator is nonpositive are $(k,j)=(0,0),$ $(k,j)=(1,0),$
and $(k,j)=(0,1),$ and these three 
pairs are excluded from the set $A.$  Thus, the supremum in \eqref{definitionOfK1} is finite and positive.

We let $K_{2}$ be an upper bound on the norm of some particular functions in a certain space, denoted $\mathcal{B}_{\rho}^{0}$ and defined in Section \ref{duchonRobertSection} below, that will be used for the analysis of the $w^n$ equation:
\begin{equation}\nonumber
\left\|\frac{16\pi^{2}}{L_{1}^{2}}\sin\left(\frac{2\pi x}{L_{1}}\right)\sin\left(\frac{4\pi x}{L_{1}}\right)\right\|_{\mathcal{B}_{\rho}^{0}}
\leq K_{2},
\qquad
\left\|\frac{16\pi^{2}}{L_{1}^{2}}\sin^{2}\left(\frac{4\pi x}{L_{1}}\right)\right\|_{\mathcal{B}_{\rho}^{0}}\leq K_{2},
\end{equation}
\begin{equation}\label{secondK2}
\left\|\frac{4\pi}{L_{1}}\sin\left(\frac{2\pi x}{L_{1}}\right)\right\|_{\mathcal{B}_{\rho}^{0}}\leq K_{2},
\qquad
\left\|\frac{8\pi}{L_{1}}\sin\left(\frac{4\pi x}{L_{1}}\right)\right\|_{\mathcal{B}_{\rho}^{0}}\leq K_{2},
\end{equation}
\begin{equation}\nonumber
\left\|\frac{16\pi^{2}}{L_{2}^{2}}\sin\left(\frac{2\pi y}{L_{2}}\right)\sin\left(\frac{4\pi y}{L_{2}}
\right)\right\|_{\mathcal{B}_{\rho}^{0}}
\leq K_{2},
\qquad
\left\|\frac{16\pi^{2}}{L_{2}^{2}}\sin^{2}\left(\frac{4\pi y}{L_{2}}\right)\right\|_{\mathcal{B}_{\rho}^{0}}\leq K_{2},
\end{equation}
\begin{equation}\nonumber
\left\|\frac{4\pi}{L_{2}}\sin\left(\frac{2\pi y}{L_{2}}\right)\right\|_{\mathcal{B}_{\rho}^{0}}\leq K_{2},
\qquad
\left\|\frac{8\pi}{L_{2}}\sin\left(\frac{4\pi y}{L_{2}}\right)\right\|_{\mathcal{B}_{\rho}^{0}}\leq K_{2}.
\end{equation}

Finally, we introduce a constant $K$ that will be used in the bound on the forcing terms:
\begin{multline}\nonumber
K=\max\Bigg\{
\frac{3M_{1,1}^{1/2}M_{3}K_{2}}{\left\|\cos\left(\frac{2\pi x}{L_{1}}\right)\right\|_{\mathcal{B}_{\rho}^{0}}},
\frac{3M_{1,1}^{1/2}M_{3}K_{2}}{\left\|\cos\left(\frac{4\pi x}{L_{1}}\right)\right\|_{\mathcal{B}_{\rho}^{0}}},
\frac{3M_{1,2}^{1/2}M_{3}K_{2}}{\left\|\cos\left(\frac{2\pi y}{L_{2}}\right)\right\|_{\mathcal{B}_{\rho}^{0}}},
\frac{3M_{1,2}^{1/2}M_{3}K_{2}}{\left\|\cos\left(\frac{4\pi y}{L_{2}}\right)\right\|_{\mathcal{B}_{\rho}^{0}}},
\\
\frac{3M_{1,2}^{1/2}M_{3}K_{2}}{\left\|\cos\left(\frac{2\pi x}{L_{1}}\right)\right\|_{\mathcal{B}_{\rho}^{0}}},
\frac{3M_{1,2}^{1/2}M_{3}K_{2}}{\left\|\cos\left(\frac{4\pi x}{L_{1}}\right)\right\|_{\mathcal{B}_{\rho}^{0}}},
\frac{3M_{1,1}^{1/2}M_{3}K_{2}}{\left\|\cos\left(\frac{2\pi y}{L_{2}}\right)\right\|_{\mathcal{B}_{\rho}^{0}}},
\frac{3M_{1,1}^{1/2}M_{3}K_{2}}{\left\|\cos\left(\frac{4\pi y}{L_{2}}\right)\right\|_{\mathcal{B}_{\rho}^{0}}}
\Bigg\}.
\end{multline}

\section{Goodman's toy model with added forcing}\label{goodmanSection}

In \cite{goodman}, Goodman proved that small solutions of the one-dimensional Kuramoto-Sivashinsky
equation exist and stay small for all time, using a Lyapunov function argument.  In his proof, the domain can be of arbitrary size, and hence there can be  any number of linearly growing modes.  First, however, he motivated the argument with a toy model, which was constructed by considering the case in which there was only one growing mode,
and neglecting contributions to the evolution from Fourier modes other than the first and second  modes.  The toy model demonstrated how energy transfers between a growing mode and a decaying mode, achieving balance.
We make two modifications to Goodman's toy model: we have a small parameter in front of the exponential growth
term in the evolution equation for the growing mode (this growth term was of unit size in \cite{goodman}), and 
we allow a given forcing as well.  In this section, we develop bounds in
Proposition \ref{lyapunovArgument} and 
Proposition \ref{enhancedBBound} 
that will be utilized 
in the induction argument in Section \ref{inductionSection} below.

We study the following system
\begin{equation}\label{aFromSystem}
a_{t}=\varepsilon_{i} a+\frac{8\pi^{2}}{L_{i}^{2}}ab+Q_{1},
\end{equation}
\begin{equation}\label{bFromSystem}
b_{t}=-B_{i}b-\frac{2\pi^{2}}{L_{i}^{2}}a^{2}+Q_{2},
\end{equation}
for $i\in\{1,2\}$, where $Q_1$ and $Q_2$ are given functions in time.  
For existence and uniqueness of solutions, at least for short time, it is enough to assume that $Q_{i}\in L^1((0,t))$. We will need a bit more hypotheses on these functions. 
For the remainder of the section, we fix a choice for $i\in\{1,2\}.$
We will assume the following bounds for $Q_{1}$ and  $Q_{2}:$ 
\begin{equation}\label{FBounds}
\sup_{t\in[0,\infty)}|Q_{1}|\leq 2K\varepsilon^{2},
\qquad
\sup_{t\in[0,\infty)}|Q_{2}|\leq 2K\varepsilon^{2}.
\end{equation}

\begin{proposition}\label{lyapunovArgument}   
Assume \eqref{FBounds} holds and let $a$ and $b$ solve \eqref{aFromSystem}-\eqref{bFromSystem}.
There exists $\varepsilon_{*}>0$  such that for any value of $\varepsilon>0$ satisfying $\varepsilon\in(0,\varepsilon_{*}),$ if $a^{2}(0)+b^{2}(0)\leq M_{1,i}\,\varepsilon/4,$ then
$a^{2}(t)+b^{2}(t)\leq 4M_{1,i}\,\varepsilon$ for all $t>0$.
\end{proposition}

\begin{proof}
We define a Lyapunov function
\begin{equation}\nonumber
G(a,b)=\frac{1}{2}a^{2}+2b^{2}+\frac{L_{i}^{2}\varepsilon}{\pi^{2}}b.
\end{equation}
Let us assume that $G(a,b)\geq M_{1,i}\varepsilon.$
Then we have that
\begin{multline}\label{workingOutGConsequence}
\frac{1}{2}a^{2}+2b^{2}
\geq M_{1,i}\,\varepsilon-\left|\frac{L_{i}^{2}\varepsilon b}{\pi^{2}}\right|
\geq M_{1,i}\,\varepsilon-b^{2}-\frac{L_{i}^{4}\varepsilon^{2}}{4\pi^{4}}
\geq \frac{M_{1,i}}{2}\varepsilon - b^{2}.
\end{multline}
For the first inequality, we have used that, by Young's inequality,
\begin{equation}\label{youngForG}
\left|\frac{L_{i}^{2}\varepsilon b}{\pi^{2}}\right| \leq b^{2}+\frac{L_{i}^{4}\varepsilon^{2}}{4\pi^{4}},
\end{equation}
while for the last inequality we have used that it is possible to choose $\varepsilon$ small enough so that
\begin{equation}\nonumber
\frac{L_{i}^{4}\varepsilon^{2}}{4\pi^{4}}\leq\frac{M_{1,i}\,\varepsilon}{2}.
\end{equation}
It then follows from  \eqref{workingOutGConsequence} that
\begin{equation}\nonumber
\frac{1}{2}a^{2}+3b^{2}\geq\frac{M_{1,i}\,\varepsilon}{2},
\end{equation}
from which we conclude that
\begin{equation}\label{abBoundedBelow}
a^{2}+b^{2}\geq\frac{M_{1,i}\,\varepsilon}{6}.
\end{equation}

We next take the derivative of $G$ with respect to time and use \eqref{aFromSystem}-\eqref{bFromSystem}:
\begin{equation}\nonumber
G_{t}=(\varepsilon_{i}-2\varepsilon) a^{2}-4B_{i}b^{2}-\frac{L_{i}^{2}\varepsilon B_{i}b}{\pi^{2}}+aQ_{1}+4bQ_{2}
+\frac{L_{i}^{2}\varepsilon}{\pi^{2}}Q_{2}.
\end{equation}
We rewrite this expression as
\begin{equation}\nonumber
G_{t}=\Upsilon_{1}+\Upsilon_{2},
\end{equation}
where $\Upsilon_{1}$ and $\Upsilon_{2}$ are given by
\begin{equation}\nonumber
\Upsilon_{1}=\left(\frac{\varepsilon_{i}}{2}-\varepsilon\right)
a^{2}-2B_{i}b^{2}-\frac{L_{i}^{2}\varepsilon B_{i}b}{\pi^{2}},
\end{equation}
\begin{equation}\nonumber
\Upsilon_{2}=\left(\frac{\varepsilon_{i}}{2}-\varepsilon\right)a^{2}-2B_{i}b^{2}
+aQ_{1}+4bQ_{2}+\frac{L_{i}^{2}\varepsilon Q_{2}}{\pi^{2}}.
\end{equation}

We will show that $\Upsilon_{1}$ and $\Upsilon_{2}$ are negative when $a$ and $b$ satisfy
\eqref{abBoundedBelow}, at least for sufficiently small values of $\varepsilon.$
For $\Upsilon_{1},$ it is enough to consider the case $b<0$, as $\Upsilon_{1}<0$ if $b\geq 0$.
Next, we observe that if $b<-\frac{L_{i}^{2}\varepsilon}{2\pi^{2}},$ then
\begin{equation}\nonumber
-2B_{i}b^{2}-\frac{L_{i}^{2}\varepsilon B_{i}b}{\pi^{2}}<0,
\end{equation}
and thus $\Upsilon_{1}<0.$  The remaining case to consider is
\begin{equation}\label{bSmallPositive}
-\frac{L_{i}^{2}\varepsilon}{2\pi^{2}}<b<0.
\end{equation}
For $\varepsilon$ small enough, \eqref{abBoundedBelow} and \eqref{bSmallPositive} together imply
\begin{equation}\nonumber
a^{2}\geq \frac{M_{1,i}\varepsilon}{12}.
\end{equation}
Hence, if \eqref{bSmallPositive} holds,  we may conclude the following bounds:
\begin{equation}\nonumber
\left(\frac{\varepsilon_{i}}{2}-\varepsilon\right)a^{2}\leq -\frac{\varepsilon}{2}a^{2}
\leq
-\frac{M_{1,i}\varepsilon^{2}}{24},
\end{equation}
\begin{equation}\nonumber
\left|\frac{L_{i}^{2}\varepsilon B_{i}b}{\pi^{2}}\right|\leq\frac{L_{i}^{4}\varepsilon^{2}B_{i}}{2\pi^{4}}.
\end{equation}
Using that $M_{1,i}=\frac{12B_{i}L_{i}^{4}}{\pi^{4}}$ by definition,
we have
\begin{equation}\nonumber
-\frac{\varepsilon}{2}a^{2}-\frac{L_{i}^{2}\varepsilon B_{i}b}{\pi^{2}}
\leq -\frac{M_{1,i}\varepsilon^{2}}{24}+\frac{L_{i}^{4}\varepsilon^{2}B_{i}}{2\pi^{4}}
=0.
\end{equation}
We conclude that $\Upsilon_{1}<0.$ We have shown then that $\Upsilon_{1}<0$ in every case.

We now turn to $\Upsilon_{2}.$ We estimate the terms containing $Q_1$ and $Q_2$ as follows.
By Young's inequality,
\begin{equation}\nonumber
|aQ_{1}|\leq\frac{\varepsilon a^{2}}{4}+\frac{Q_{1}^{2}}{\varepsilon},
\end{equation}
which, combined with \eqref{FBounds}, gives
\begin{equation}\nonumber
|aQ_{1}|\leq\frac{\varepsilon a^{2}}{4}+\frac{4K^{2}\varepsilon^{4}}{\varepsilon}=\frac{\varepsilon a^{2}}{4}
+4K^{2}\varepsilon^{3}.
\end{equation}
We similarly bound $4bQ_{2}$ as
\begin{equation}\nonumber
|4bQ_{2}|\leq B_{i}b^{2}+\frac{4Q_{2}^{2}}{B_{i}}\leq B_{i}b^{2}+\frac{16K^{2}\varepsilon^{4}}{B_{i}}.
\end{equation}
Again using \eqref{FBounds}, we bound the last term in $\Upsilon_2$ as
\begin{equation}\nonumber
\left|\frac{L_{i}^{2}\varepsilon Q_{2}}{\pi^{2}}\right| \leq \frac{2L_{i}^{2}K\varepsilon^{3}}{\pi^{2}}.
\end{equation}
These estimates in turn give  the following bound on $\Upsilon_2$:
\begin{equation}
\Upsilon_{2}\leq -\frac{\varepsilon a^{2}}{4}-B_{i}b^{2}
+\left[
4K^{2}\varepsilon^{3}+\frac{16K^{2}\varepsilon^{4}}{B_{i}}
+\frac{2L_{i}^{2}K\varepsilon^{3}}{\pi^{2}}\right].
\end{equation}
But we assumed that $G\geq M_{1,i}\varepsilon$, which implies
$a^{2}+b^{2}\geq M_{1,i}\varepsilon/6$ as shown above, so that
\begin{equation}\nonumber
-\frac{\varepsilon a^{2}}{4}-B_{i}b^{2} \leq -\frac{\varepsilon a^{2}}{4}-\frac{\varepsilon b^{2}}{4}\leq 
-\frac{M_{1,i}\varepsilon^{2}}{24}.
\end{equation}
Therefore,
\begin{equation}\nonumber
\Upsilon_{2}< -\frac{M_{1,i}\varepsilon^{2}}{24}
+\left[
4K^{2}\varepsilon^{3}+\frac{16K^{2}\varepsilon^{4}}{B_{i}}
+\frac{2L_{i}^{2}K\varepsilon^{3}}{\pi^{2}}\right].
\end{equation}
We can take $\varepsilon$ small enough so that
\begin{equation}\nonumber
\left|
4K^{2}\varepsilon^{3}+\frac{16K^{2}\varepsilon^{4}}{B_{i}}
+\frac{2L_{i}^{2}K\varepsilon^{3}}{\pi^{2}}\right|
\leq \frac{M_{1,i}\varepsilon^{2}}{48}.
\end{equation}
For such values of $\varepsilon$  we have $\Upsilon_{2}<0.$

We have concluded that $G\geq M_{1,i}\varepsilon$ implies $G_{t}<0.$  Hence, if $G$ is initially less than
$M_{1,i}\,\varepsilon,$ then necessarily $G<M_{1,i}\,\varepsilon$ for all $t>0$.

Next, we ask under which conditions  $G<M_{1,i}\varepsilon$ initially.
We observe that,  from the definition of $G$ and \eqref{youngForG},
\begin{equation}\nonumber
G\leq\frac{1}{2}a^{2}+2b^{2}+b^{2}+\frac{L_{i}^{4}\varepsilon^{2}}{4\pi^{4}}\leq 3(a^{2}+b^{2})
+\frac{L_{i}^{4}\varepsilon^{2}}{4\pi^{4}}.
\end{equation}
Consequently,  $G(0)<M_{1,i}\varepsilon$ provided $a^{2}(0)+b^{2}(0)\leq\frac{M_{1,i}\varepsilon}{4}$ (which holds 
by hypothesis) and provided $\varepsilon$ is taken small enough so that
$\frac{L_{i}^{4}\varepsilon^{2}}{4\pi^{4}}<\frac{M_{1,i}\varepsilon}{4}$.

Assuming then $G(t)<M_{1,i}\varepsilon$  for all $t>0$, we ask what can we say about $a^{2}(t)+b^{2}(t)$.
We again use the definition of $G$ together with
\eqref{youngForG}, now finding that
\begin{multline}\nonumber
M_{1,i}\varepsilon>\frac{1}{2}a^{2}+2b^{2}+\frac{L_{i}^{2}\varepsilon b}{\pi^{2}}
\geq \frac{1}{2}a^{2}+2b^{2}-\left|\frac{L_{i}^{2}\varepsilon b}{\pi^{2}}\right|
\\
\geq \frac{1}{2}a^{2}+2b^{2}-b^{2}-\frac{L_{i}^{4}\varepsilon^{2}}{4\pi^{4}}
\geq\frac{1}{2}\left(a^{2}+b^{2}\right)-\frac{L_{i}^{4}\varepsilon^{2}}{4\pi^{4}}.
\end{multline}
Rearranging the left-hand and right-hand sides of this expression gives
\begin{equation}\nonumber
\frac{1}{2}\left(a^{2}+b^{2}\right) < M_{1,i}\varepsilon+\frac{L_{i}^{4}\varepsilon^{2}}{4\pi^{4}}.
\end{equation}
We then take $\varepsilon$ small enough so that
$\frac{L_{i}^{4}\varepsilon^{2}}{4\pi^{4}}\leq M_{1,i}\varepsilon.$
Finally, we conclude
\begin{equation}\nonumber
a^{2}+b^{2} < 4M_{1,i}\varepsilon.
\end{equation}
This completes the proof.
\end{proof}

\begin{proposition}\label{enhancedBBound}
Under the hypotheses of Proposition \ref{lyapunovArgument}, if also
$b(0)\leq M_{2,i}\varepsilon/2,$ then there exists $\varepsilon_{*}>0$ such that
for any value of $\varepsilon\in(0,\varepsilon_{*}),$  $|b(t)|\leq M_{2,i}\,\varepsilon$ for all $t>0$.
\end{proposition}

\begin{proof}
From Proposition \ref{lyapunovArgument}, we have $(a(t))^{2}\leq 4M_{1,i}\,\varepsilon$  for all $t>0$.
From \eqref{FBounds}, we also have
$|Q_{2}(t)|\leq 2K\varepsilon^{2}$ for all $t>0$.
Using  Duhamel's Formula, we rewrite the equation for $b$ in integral form:
\begin{equation}\nonumber
b(t)=e^{-B_{i}t}b(0)+e^{-B_{i}t}\int_{0}^{t}e^{B_{i}s}\left[\frac{2\pi^{2}}{L_{i}^{2}}a^{2}(s)+Q_{2}(s)\right]
\ ds.
\end{equation}
We recall that $M_{2,i}=\frac{8\pi^{2}M_{1,i}}{L_{i}^{2}}$, so that
\begin{equation}\nonumber
|b(t)|\leq e^{-B_{i}t}|b(0)|+\left(M_{2,i}\varepsilon+2K\varepsilon^{2}\right)e^{-B_{i}t}\int_{0}^{t}e^{B_{i}s}\ ds.
\end{equation}
We evaluate the integral and bound the result as
\begin{equation}\nonumber
|b(t)|\leq e^{-B_{i}t}|b(0)|+\frac{1}{B_{i}}\left(M_{2,i}\varepsilon+2K\varepsilon^{2}\right).
\end{equation}
Now $B_{i}$ is approximately equal to $12$, since we are taking $L_{i}$
close to $2\pi;$ we may thus say $B_{i}>10.$  
Lastly, by again taking $\varepsilon$ sufficiently small and from the hypothesis 
$|b(0)|\leq \frac{M_{2,i}}{2}\varepsilon$, it follows that
\begin{equation}\nonumber
|b(t)|\leq M_{2,i}\varepsilon,
\end{equation}
for all $t>0.$
\end{proof}

\section{The Duchon-Robert framework}\label{duchonRobertSection}

In this section we develop the estimates we will use for the iterates $w^{n}.$  We will assume that $w^n$ belongs to suitable function spaces of analytic functions in time based on the Wiener algebra. These spaces are Banach algebras and are well adapted to the inductive argument carried out in Section \ref{inductionSection}. The bounds on $w^n$ follows from estimates on the semigroup generated by the linearized operator and by estimating the integral in the mild formulation of the PDE, exploiting the algebra structure to control the nonlinearity.  These spaces and similar bounds were used by Duchon and Robert \cite{duchonRobert} to prove the global existence of vortex sheet solutions in incompressible two-dimensional fluid flow.

For $m\in\mathbb{N}$ and $\rho\geq 0,$ we define the space $B^{m}_{\rho}$ to be the space of distributions on the torus for for which the following weighted sum of  their Fourier coefficients is finite:
\begin{equation}\nonumber
f\in B^{m}_{\rho} \iff \|f\|_{B^{m}_{\rho}}=\sum_{(k,j)\in\mathbb{Z}^{2}}e^{\rho(|k|+|j|)}(1+|k|+|j|)^{m}|f_{k,j}|<\infty.
\end{equation}
We also have a space-time version of this space, which we call $\mathcal{B}^{m}_{\rho}$, defined as the space of distributions on $[0,\infty)\times \mathbb{T}^2$ such that
\begin{equation}\nonumber
   \|g\|_{\mathcal{B}^{m}_{\rho}}
=\sum_{(k,j)\in\mathbb{Z}^{2}}e^{\rho(|k|+|j|)}(1+|k|+|j|)^{m}\sup_{t\in[0,\infty)}|g_{k,j}(t)|<\infty.
\end{equation}
We observe that elements of both $B^m_\rho$ and $\mathcal{B}^m_\rho$ are actually functions that are analytic  in space with radius of analyticity at least $\rho>0$ and at least bounded in time.

The spaces $B^{0}_{\rho}$ and $\mathcal{B}^{0}_{\rho}$ are Banach algebras;
indeed, $B_{0}^{0}$ is exactly the Wiener algebra.  If $f$ and $g$ are both in 
${\mathcal{B}}^{0}_{\rho},$ we have
\begin{multline}\nonumber
\|fg\|_{{\mathcal{B}}^{0}_{\rho}}
\\
\leq \sum_{(k,j)\in\mathbb{Z}^{2}}
\sum_{(\ell,n)\in\mathbb{Z}}e^{\rho(|k-\ell|+|j-n|)}e^{\rho(|\ell|+|n|)}
\left(\sup_{t\in[0,\infty)}|f_{k-\ell,j-n}(t)|\right)\left(\sup_{t\in[0,\infty)}|g_{\ell,n}(t)|\right)
\\
 \leq \|f\|_{{\mathcal{B}}^0_{\rho}}\|g\|_{{\mathcal{B}}^0_{\rho}}.
\end{multline}
The analogous estimate for $B_{\rho}^{0}$ follows immediately by observing that $B_{\rho}^{m}$ consists precisely of the elements of $\mathcal{B}_{\rho}^{m}$ that are constant in time.
Then we may conclude (simply by the product rule) that the spaces $B^{1}_{\rho}$ and $\mathcal{B}^{1}_{\rho}$ 
are also Banach algebras.  Indeed,  a function $f$ is in $\mathcal{B}^{1}_{\rho}$ if and only if $f$ and its partial derivatives $\partial_{x}f$ and $\partial_{y}f$ are all in $\mathcal{B}_{\rho}^{0}.$

We note that we will not use the spaces $\mathcal{B}^{m}_{\rho}$ or $B^{m}_{\rho}$ for $m>1,$ although these
are Banach algebras as well (for the same reasons).

We define the operator $I^{+}$ by
\begin{equation}\nonumber
I^{+}h(\cdot,t)=\mathbb{P}_{5}\int_{0}^{t}e^{-(\Delta^{2}+\Delta)(t-s)}h(\cdot,s)\ ds,
\end{equation}
where the integral is intended in the Bochner sense and $e^{-t (\Delta^2+ \Delta)}$ denotes the $C^0$ (unbounded) semigroup generated by the linearized KSE operator on  $B_{\rho}^{m}$.
We will show that $I^{+}$ is  bounded from $\mathcal{B}^{0}_{\rho}$ to $\mathcal{B}^{1}_{\rho}.$
(This is the only fact needed for our purposes, but the integral is actually bounded from $\mathcal{B}^{0}_{\rho}$ to
$\mathcal{B}^{4}_{\rho}.$)
Let $h\in\mathcal{B}^{0}_{\rho}$ be given.  Then the norm of $I^{+}h$ is given by
\begin{multline}\nonumber
\|I^{+}h\|_{\mathcal{B}^{1}_{\rho}}
\\
=\sum_{(k,j)\notin A}e^{\rho(|k|+|j|)}(1+|k|+|j|)\sup_{t\in[0,\infty)}\left|
\int_{0}^{t}\exp\left\{\sigma(k,j)(t-s)\right\}
h_{k,j}(s)\ ds
\right|,
\end{multline}
where $\sigma$ is defined in \eqref{eq:sigmaDef}.
The triangle inequality then implies
\begin{multline}\nonumber
\|I^{+}h\|_{\mathcal{B}^{1}_{\rho}}
\leq
\sum_{(k,j)\notin A}
e^{\rho(|k|+|j|)}(1+|k|+|j|)\sup_{t\in[0,\infty)}\exp\left\{\sigma(k,j)t\right\}
\cdot
\\
\cdot\int_{0}^{t}\exp\left\{-\sigma(k,j)s\right\}
|h_{k,j}(s)|\ ds.
\end{multline}
We take the supremum of $|h_{k,j}(s)|$ in $s$, which we can then pull out to obtain:
\begin{multline}\nonumber
\|I^{+}h\|_{\mathcal{B}^{1}_{\rho}}\leq\left(\sum_{(k,j)\notin A}e^{\rho(|k|+|j|)}\sup_{t\in[0,\infty)}|h_{k,j}(t)|\right)
\\
\Bigg[\sup_{t\in[0,\infty)}\sup_{(k,j)\notin A}\Bigg((1+|k|+|j|)
\exp\left\{\sigma(k,j)t\right\}\int_{0}^{t}\exp\left\{-\sigma(k,j)s\right\}\ ds
\Bigg)\Bigg].
\end{multline}
The first factor on the right-hand side can simply be bounded by $\|h\|_{\mathcal{B}^{0}_{\rho}}.$
A  bound on the second factor (i.e., the double supremum) can be found by  directly computing the integral, which gives:
\begin{equation}\nonumber
\|I^{+}h\|_{\mathcal{B}^{1}_{\rho}}\leq \|h\|_{\mathcal{B}^{0}_{\rho}}
\left[\sup_{t\in[0,\infty)}\sup_{(k,j)\notin A}
\frac{(1+|k|+|j|)(1-\exp\{\sigma(k,j)t\}}
{-\sigma(k,j)}\right].
\end{equation}
The negative term in the numerator can be neglected.  Therefore, we have
\begin{equation}\nonumber
\|I^{+}h\|_{\mathcal{B}^{1}_{\rho}}\leq K_{1}\|h\|_{\mathcal{B}^{0}_{\rho}},
\end{equation}
where 
\begin{equation}\nonumber
K_{1}=\sup_{(k,j)\notin A}\frac{1+|k|+|j|}{-\sigma(k,j)}.
\end{equation}

We now turn to proving estimates on the semigroup.  We show that  $e^{(-\Delta^{2}-\Delta)t}$ maps
$\mathbb{P}_5(B^{1}_{\rho})$ into $\mathcal{B}^{1}_{\rho}$ boundedly. In fact, we first observe that
\begin{multline}\nonumber
\|e^{(-\Delta^{2}-\Delta)t} f\|_{\mathcal{B}_{\rho}^{1}}
\\
\leq \sum_{(k,j)\notin A}e^{\rho(|k|+|j|)}(1+|k|+|j|)\sup_{t\in[0,\infty)}
\exp\left\{\sigma(k,j)t\right\}|(f)_{k,j}|,
\end{multline}
if $f\in \mathbb{P}_5(B^{1}_{\rho})$.
The supremum is achieved at $t=0$ for every $(k,j)\notin A.$
(Recall that $f$ is supported in Fourier space only on wavenumbers  in the complement of the set $A.$)
We therefore have
\begin{equation}\label{eq:semigroupEst}
 \|e^{(-\Delta^{2}-\Delta)t} f\|_{\mathcal{B}^{1}_{\rho}}
\leq \sum_{(k,j)\notin A}e^{\rho(|k|+|j|)}(1+|k|+|j|)|(f)_{k,j}| = \|f\|_{B^{1}_{\rho}}.
\end{equation}
We will apply these semigroup estimates to the remainder terms $w^n$.

\section{Inductive argument and convergence}\label{inductionSection}

We are now ready to complete the proof of Theorem \ref{nontechnicalTheorem}.  First, in Proposition \ref{uniformEstimate} we obtain uniform bounds on the iterates by induction, using the bounds already established.  Then  in Theorem \ref{mainTheoremTechnical}, we state a precise version of our main result, existence of a global mild solution $\phi$, which follows by passing to the limit $n\to\infty$ and using compactness arguments.

\begin{proposition}\label{uniformEstimate}
Fix $\rho>0.$ Let $\varepsilon=\max\{\varepsilon_{1},\varepsilon_{2}\}$, where $\varepsilon_i$, $i=1,2$, is given in \eqref{eq:epBdef}.
Assume the initial data $a_{1,0}(0),$ $a_{2,0}(0),$ $a_{0,1}(0),$ $a_{0,2}(0),$ and $w(0)$  satisfy
\begin{equation}\nonumber
(a_{1,0}(0))^{2}+(a_{2,0}(0))^{2}\leq \frac{M_{1,1}\varepsilon}{4},
\end{equation}
\begin{equation}\nonumber
(a_{0,1}(0))^{2}+(a_{0,2}(0))^{2}\leq \frac{M_{1,2}\varepsilon}{4},
\end{equation}
\begin{equation}\nonumber
|a_{2,0}(0)|\leq\frac{M_{2,1}\varepsilon}{2},
\qquad |a_{0,2}(0)|\leq \frac{M_{2,2}\varepsilon}{2},
\end{equation}
\begin{equation}\nonumber
\|w(0)\|_{B_{\rho}^{1}}\leq \frac{M_{3}}{6}\varepsilon^{3/2}.
\end{equation}
Then there exists $\varepsilon_{*}>0$ such that for  $i\in\{1,2\},$
for all $\varepsilon_{i}\in(0,\varepsilon_{*}),$ and for all $n,$ the following bounds are satisfied:
\begin{equation}\label{anInductive}
\sup_{t\in[0,\infty)}|a^{n}_{1,0}|\leq 2M_{1,1}^{1/2}\varepsilon^{1/2},
\qquad
\sup_{t\in[0,\infty)}|a^{n}_{0,1}|\leq 2M_{1,2}^{1/2}\varepsilon^{1/2},
\end{equation}
\begin{equation}\label{bnInductive}
\sup_{t\in[0,\infty)}|a^{n}_{2,0}|\leq M_{2,1}\varepsilon,
\qquad
\sup_{t\in[0,\infty)}|a^{n}_{0,2}|\leq M_{2,2}\varepsilon,
\end{equation}
\begin{equation}\label{wnInductive}
\|w^{n}\|_{\mathcal{B}_{\rho}^{1}}\leq M_{3}\varepsilon^{3/2}.
\end{equation}
\begin{equation}\label{firstFInductive}
\sup_{t\in[0,\infty)}|F_{1,0,x}^{n}|\leq K\varepsilon^{2},
\qquad
\sup_{t\in[0,\infty)}|F_{1,0,y}^{n}|\leq K\varepsilon^{2},
\end{equation}
\begin{equation}\label{secondFInductive}
\sup_{t\in[0,\infty)}|F_{2,0,x}^{n}|\leq K\varepsilon^{2},
\qquad
\sup_{t\in[0,\infty)}|F_{2,0,y}^{n}|\leq K\varepsilon^{2}.
\end{equation}
\begin{equation}\label{thirdFInductive}
\sup_{t\in[0,\infty)}|F_{0,1,x}^{n}|\leq K\varepsilon^{2},
\qquad
\sup_{t\in[0,\infty)}|F_{0,1,y}^{n}|\leq K\varepsilon^{2},
\end{equation}
\begin{equation}\label{fourthFInductive}
\sup_{t\in[0,\infty)}|F_{0,2,x}^{n}|\leq K\varepsilon^{2},
\qquad
\sup_{t\in[0,\infty)}|F_{0,2,y}^{n}|\leq K\varepsilon^{2}.
\end{equation}

\end{proposition}

\begin{proof}
We initialize our iterative scheme with $a_{1,0}^{0}=a_{2,0}^{0}=0,$
$a_{0,1}^{0}=a_{0,2}^{0}(0)=0,$ and $w^{0}=0.$
The bounds \eqref{anInductive}, \eqref{bnInductive}, \eqref{wnInductive},
\eqref{firstFInductive}, \eqref{secondFInductive}, \eqref{thirdFInductive}, and \eqref{fourthFInductive}
are trivially satisfied by $a_{1,0}^{0},$ $a_{2,0}^{0},$ $a_{0,1}^{0},$ $a_{0,2}^{0},$ and $w^{0}.$  
We assume \eqref{anInductive}, \eqref{bnInductive},
\eqref{wnInductive}, \eqref{firstFInductive}, \eqref{secondFInductive}, \eqref{thirdFInductive}, 
and \eqref{fourthFInductive}, as our inductive hypothesis.
We now prove the analogues of these for the next iterate.  We recall the definition of $\phi^{n}$ from
\eqref{iteratedPhi}.

An appeal to Proposition \ref{lyapunovArgument} with $i=1$ 
immediately proves the desired bound on $a_{1,0}^{n+1},$ and another appeal to Proposition \ref{lyapunovArgument}
with $i=2$ immediately proves the desired bound on $a_{0,1}^{n+1}.$
Then appealing twice to Proposition \ref{enhancedBBound} 
again immediately proves the desired bounds on $a_{2,0}^{n+1}$ and $a_{0,2}^{n+1}.$

Next, we write the mild formulation of the equation for $w^{n+1}$ from \eqref{iteratedW}.  Since $\mathbb{P}_{5}$ is a projection, we
may
write $\mathbb{P}_{5}=\mathbb{P}_{5}^{2}.$  Using the definition of $I^{+}$ introduced in
Section \ref{duchonRobertSection}, we have 
\begin{equation}\nonumber
w^{n+1}=e^{(-\Delta^{2}-\Delta)t}w_{0}+I^{+}(\mathbb{P}_{5}((\phi_{x}^{n})^{2}+(\phi_{y}^{n})^{2})).
\end{equation}
Using the bounds developed in Section \ref{duchonRobertSection}, we can then estimate $w^{n+1}$ as follows:
\begin{equation}\label{aboutToBoundW}
\|w^{n+1}\|_{\mathcal{B}_{\rho}^{1}} \leq \|w_{0}\|_{B_{\rho}^{1}}
+K_{1}\left\|\mathbb{P}_{5}((\phi_{x}^{n})^{2})\right\|_{\mathcal{B}_{\rho}^{0}}
+K_{1}\left\|\mathbb{P}_{5}((\phi_{y}^{n})^{2})\right\|_{\mathcal{B}_{\rho}^{0}}.
\end{equation}

In order to close the induction argument,  we need to express $\phi_{x}^{n}$ and $\phi_{y}^{n}$ in terms of the quantities we are estimating. To this end,  we will use
a different decomposition for $(\phi_{x}^{n})^{2}$ and $(\phi_{y}^{n})^{2}$ than the one used in Section \ref{fSection}.
We decompose $(\phi_{x}^{n})^{2}$ in the following way:
\begin{equation}\label{newSquareDecomp}
(\phi_{x}^{n})^{2}=\Phi_{0}+\Phi_{1}+\Phi_{2}w_{x}^{n}+(w_{x}^{n})^{2},
\end{equation}
where $\Phi_{0},$ $\Phi_{1},$ and $\Phi_{2}$ are given by
\begin{equation}\nonumber
\Phi_{0}=\frac{4\pi^{2}(a_{1,0}^{n})^{2}}{L_{1}^{2}}\sin^{2}\left(\frac{2\pi x}{L_{1}}\right),
\end{equation}
\begin{equation}\nonumber
\Phi_{1}=\frac{16\pi^{2}a_{1,0}^{n}a_{2,0}^{n}}{L_{1}^{2}}
\sin\left(\frac{2\pi x}{L_{1}}\right)\sin\left(\frac{4\pi x}{L_{1}}\right)
+\frac{16\pi^{2}(a_{2,0}^{n})^{2}}{L_{1}^{2}}\sin^{2}\left(\frac{4\pi x}{L_{1}}\right),
\end{equation}
\begin{equation}\nonumber
\Phi_{2}=-\frac{4\pi a_{1,0}^{n}}{L_{1}}\sin\left(\frac{2\pi x}{L_{1}}\right)
-\frac{8\pi a_{2,0}^{n}}{L_{1}}\sin\left(\frac{4\pi x}{L_{1}}\right).
\end{equation}
One reason for this decomposition is that $\mathbb{P}_{5}\Phi_{0}=0.$  Another reason is that the term
$\Phi_{0}$ is larger than the remaining terms;  $\Phi_{1},$ $\Phi_{2}w^{n}_{x},$ and $(w_{x}^{n})^{2}$
are all of order $\varepsilon^{3/2}$ or smaller, while the same is not true for $\Phi_{0}.$

We now estimate $\left\|\mathbb{P}_{5}((\phi_{x}^{n})^{2})\right\|_{\mathcal{B}_{\rho}^{0}},$ using the fact that $\mathbb{P}_{5}\Phi_{0}=0$ and the fact $\mathbb{P}_{5}$ is bounded:
\begin{equation}\label{phixsquaredDecompositionBound}
\left\|\mathbb{P}_{5}((\phi_{x}^{n})^{2})\right\|_{\mathcal{B}_{\rho}^{0}}
\leq \|\Phi_{1}^{n}\|_{\mathcal{B}_{\rho}^{0}} + \|\Phi_{2}^{n}\|_{\mathcal{B}_{\rho}^{0}}\|w^{n}\|_{\mathcal{B}_{\rho}^{1}}
+\|w^{n}\|_{\mathcal{B}_{\rho}^{1}}^{2}.
\end{equation}
Above, we have also used the algebra property for $\mathcal{B}_{\rho}^{0}$ and that
$\|w_{x}^{n}\|_{\mathcal{B}_{\rho}^{0}}\leq \|w^{n}\|_{\mathcal{B}_{\rho}^{1}},$ which is a direct consequence of the definition.

We can then make some straightforward estimates of $\Phi_{1}$ and $\Phi_{2}.$  For $\Phi_{1}$ we have
\begin{equation}\nonumber
\|\Phi_{1}\|_{\mathcal{B}_{\rho}^{0}}
\leq 2M_{1,1}^{1/2}M_{2,1}K_{2}\varepsilon^{3/2}+M_{2,1}^{2}K_{2}\varepsilon^{2}.
\end{equation}
Of course, to get this bound, we have employed the inductive hypothesis and the fact  that $
\varepsilon_{1}\leq\varepsilon.$
We recall that $M_{3}\leq6K_{1}(2M_{1,1}^{1/2}M_{2,1}K_{2})$ to conclude that
\begin{equation}\nonumber
\|\Phi_{1}\|_{\mathcal{B}_{\rho}^{0}}\leq \frac{M_{3}}{6K_{1}}\varepsilon^{3/2}
+M_{2,1}^{2}K_{2}\varepsilon^{2}.
\end{equation}
We take $\varepsilon$ small enough so that
\begin{equation}\nonumber
M_{2,1}^{2}K_{2}\varepsilon^{2}\leq \frac{M_{3}}{6K_{1}}\varepsilon^{3/2}.
\end{equation}
We thus have
\begin{equation}\label{oneM3}
\|\Phi_{1}\|_{\mathcal{B}_{\rho}^{0}}\leq \frac{M_{3}}{3K_{1}}\varepsilon^{3/2}.
\end{equation}

We next turn to bounding $\Phi_{2}.$  By using again the inductive hypothesis and the definition of the constant $K_{2},$
it follows that
\begin{equation}\nonumber
\|\Phi_{2}\|_{\mathcal{B}_{\rho}^{0}}\leq 2M_{1,1}^{1/2}K_{2}\varepsilon^{1/2}
+M_{2,1}K_{2}\varepsilon.
\end{equation}
Another application of the inductive hypothesis gives that
\begin{equation}\nonumber
\|\Phi_{2}\|_{\mathcal{B}_{\rho}^{0}}\|w^{n}\|_{\mathcal{B}_{\rho}^{1}}
\leq 2M_{1,1}^{1/2}M_{3}K_{2}\varepsilon^{2}+M_{2,1}M_{3}K_{2}\varepsilon^{5/2}.
\end{equation}
We take $\varepsilon$ small enough so that
\begin{equation}\nonumber
2M_{1,1}^{1/2}M_{3}K_{2}\varepsilon^{2}+M_{2,1}M_{3}K_{2}\varepsilon^{5/2}
\leq\frac{M_{3}}{24K_{1}}\varepsilon^{3/2}.
\end{equation}
We then have
\begin{equation}\label{twoM3}
\|\Phi_{2}\|_{\mathcal{B}_{\rho}^{0}}\|w^{n}\|_{\mathcal{B}_{\rho}^{1}}\leq \frac{M_{3}}{24K_{1}}\varepsilon^{3/2}.
\end{equation}

We proceed in a similar fashion to bound the quadratic term:
\begin{equation}\nonumber
\|w^{n}\|_{\mathcal{B}_{\rho}^{1}}^{2}\leq M_{3}^{2}\varepsilon^{3}.
\end{equation}
We take $\varepsilon$ small enough so that
\begin{equation}\label{M3SquaredBound}
M_{3}^{2}\varepsilon^{3}\leq \frac{M_{3}}{24K_{1}}\varepsilon^{3/2},
\end{equation}
which implies
\begin{equation}\label{threeM3}
\|w^{n}\|_{\mathcal{B}_{\rho}^{1}}^{2}\leq \frac{M_{3}}{24K_{1}}\varepsilon^{3/2}.
\end{equation}

Having concluded our treatment of $(\phi_{x}^{n})^{2},$ we now consider $(\phi_{y}^{n})^{2}.$
We may treat $(\phi_{y}^{n})^{2}$ analogously to the way we treated $(\phi_{x}^{n})^{2}$ in 
\eqref{newSquareDecomp}, leading to the decomposition
\begin{equation}\nonumber
(\phi_{y}^{n})^{2}=\Phi_{3}+\Phi_{4}+\Phi_{5}w_{y}^{n}+(w_{y}^{n})^{2},
\end{equation}
with the formulas
\begin{equation}\nonumber
\Phi_{3}=\frac{4\pi^{2}(a_{0,1}^{n})^{2}}{L_{2}^{2}}\sin^{2}\left(\frac{2\pi y}{L_{2}}\right),
\end{equation}
\begin{equation}\nonumber
\Phi_{4}=\frac{16\pi^{2}a_{0,1}^{n}a_{0,2}^{n}}{L_{2}^{2}}
\sin\left(\frac{2\pi y}{L_{2}}\right)\sin\left(\frac{4\pi y}{L_{2}}\right)
+\frac{16\pi^{2}(a_{0,2}^{n})^{2}}{L_{2}^{2}}\sin^{2}\left(\frac{4\pi y}{L_{2}}\right),
\end{equation}
\begin{equation}\nonumber
\Phi_{5}=-\frac{4\pi a_{0,1}^{n}}{L_{2}}\sin\left(\frac{2\pi y}{L_{2}}\right)
-\frac{8\pi a_{0,2}^{n}}{L_{2}}\sin\left(\frac{4\pi y}{L_{2}}\right).
\end{equation}
Similarly to the case for $\Phi_{0},$ we have $\mathbb{P}_{5}\Phi_{3}=0.$
We may then bound $(\phi_{y}^{n})^{2}$ as
\begin{equation}\label{newM3Goal}
\|\mathbb{P}_{5}(\phi_{y}^{n})^{2}\|_{\mathcal{B}_{\rho}^{0}}\leq\|\Phi_{4}\|_{\mathcal{B}_{\rho}^{0}}
+\|\Phi_{5}\|_{\mathcal{B}_{\rho}^{0}}\|w^{n}\|_{\mathcal{B}_{\rho}^{1}}+\|w^{n}\|_{\mathcal{B}_{\rho}^{1}}^{2}.
\end{equation}
We next proceed to estimating the remaining terms in a manner analogous to the previous case, omitting details:
\begin{equation}\label{newM3-1}
\|\Phi_{4}\|_{\mathcal{B}_{\rho}^{0}}\leq \frac{M_{3}}{3K_{1}}\varepsilon^{3/2},
\end{equation}
\begin{equation}\label{newM3-2}
\|\Phi_{5}\|_{\mathcal{B}_{\rho}^{0}}\|w^{n}\|_{\mathcal{B}_{\rho}^{1}}\leq\frac{M_{3}}{24K_{1}}\varepsilon^{3/2},
\end{equation}
where $\varepsilon$ is chosen sufficiently small.

We recall the condition on the initial data for $w,$ namely,
\begin{equation}\label{fiveM3}
\|w_{0}\|_{B_{\rho}^{1}}\leq \frac{M_{3}}{6}\varepsilon^{3/2}.
\end{equation}
We then combine \eqref{aboutToBoundW}, \eqref{phixsquaredDecompositionBound}, \eqref{oneM3}, \eqref{twoM3}, \eqref{threeM3}, \eqref{newM3Goal},
\eqref{newM3-1}, \eqref{newM3-2}, 
and \eqref{fiveM3} to conclude
\begin{equation}
\|w^{n+1}\|_{\mathcal{B}_{\rho}^{1}}\leq M_{3}\varepsilon^{3/2}.
\end{equation}

It remains to demonstrate the estimates for the forcing terms. We include a proof only for
$F^{n+1}_{1,0,x}$, as the other cases are similar.

Recalling the formulas of Section \ref{fSection}, we have the decomposition for $F^{n+1}_{1,0,x}$ 
\begin{equation}\nonumber
F_{1,0,x}^{n+1}\cos\left(\frac{2\pi x}{L_{1}}\right) = \mathbb{P}_{1,0}\left[\Psi_{3}^{n+1}+\Psi_{5}^{n+1}+\Psi_{6}^{n+1}\right],
\end{equation}
where
\begin{equation}\nonumber
\Psi_{3}^{n+1}=(w_{x}^{n+1})^{2},
\end{equation}
\begin{equation}\nonumber
\Psi_{5}^{n+1}=-\frac{4\pi a_{1,0}^{n+1}}{L_{1}}w_{x}^{n+1}\sin\left(\frac{2\pi x}{L_{1}}\right),
\end{equation}
\begin{equation}\nonumber
\Psi_{6}^{n+1}=-\frac{8\pi a_{2,0}^{n+1}}{L_{1}}w_{x}^{n+1}\sin\left(\frac{4\pi x}{L_{1}}\right).
\end{equation}
We may then estimate $F^{n+1}_{1,0,x}$ as 
\begin{equation}\label{f10xBound}
|F_{1,0,x}^{n+1}|\leq\frac{\|\Psi_{3}^{n+1}\|_{\mathcal{B}_{\rho}^{0}}+\|\Psi_{5}^{n+1}\|_{\mathcal{B}_{\rho}^{0}}
+\|\Psi_{6}^{n+1}\|_{\mathcal{B}_{\rho}^{0}}}{\left\|\cos\left(\frac{2\pi x}{L_{1}}\right)\right\|_{\mathcal{B}_{\rho}^{0}}}.
\end{equation}
For $\Psi_{3}^{n+1}$ we have the estimate
\begin{equation}\label{psi3Bound}
\|\Psi_{3}^{n+1}\|_{\mathcal{B}_{\rho}^{0}}\leq\|w_{x}^{n+1}\|_{\mathcal{B}_{\rho}^{0}}^{2}
\leq \|w^{n+1}\|_{\mathcal{B}_{\rho}^{1}}^{2}\leq M_{3}^{2}\varepsilon^{3}.
\end{equation}
For $\Psi_{5}^{n+1}$ we use \eqref{secondK2}, as well as the inductive hypothesis, finding
\begin{equation}\label{psi5Bound}
\|\Psi_{5}^{n+1}\|_{\mathcal{B}_{\rho}^{0}}\leq 2M_{1,1}^{1/2}M_{3}K_{2}\varepsilon^{2}.
\end{equation}
The next term, $\Psi_{6}^{n+1},$ is similar, and we bound it as
\begin{equation}\label{psi6Bound}
\|\Psi_{6}^{n+1}\|_{\mathcal{B}_{\rho}^{0}}\leq M_{2,1}M_{3}K_{2}\varepsilon^{5/2}.
\end{equation}
From the definition of $K$ we have the bound
\begin{equation}\nonumber
\frac{3M_{1,1}^{1/2}M_{3}K_{2}}{\left\|\cos\left(\frac{2\pi x}{L_{1}}\right)\right\|_{\mathcal{B}_{\rho}^{0}}}\leq K.
\end{equation}
By exploiting this bound, combining \eqref{f10xBound}, \eqref{psi3Bound}, \eqref{psi5Bound}, and \eqref{psi6Bound}, and taking $\varepsilon$ sufficiently small, we have
\begin{equation}\nonumber
|F_{1,0,x}^{n+1}|\leq K\varepsilon^{2}.
\end{equation}

We omit the details of the bounds for $F^{n+1}_{2,0,x},$ $F^{n+1}_{1,0,y},$ $F^{n+1}_{2,0,y},$
$F^{n+1}_{0,1,x},$ $F^{n+1}_{0,1,y},$ $F^{n+1}_{0,2,x},$ and
$F^{n+1}_{0,2,y}.$  Therefore, this completes the proof.
\end{proof}

We now state our main theorem.

\begin{theorem}\label{mainTheoremTechnical}  Let $\bar{\psi}_{0}\in\mathbb{R}$ be given. Fix $0<T<\infty$.
Let $\varepsilon_{i}=-\left(\frac{2\pi}{L_{i}}\right)^{4}+\left(\frac{2\pi}{L_{i}}\right)^{2}$, and set
$\varepsilon=\max\{\varepsilon_{1}, \varepsilon_{2}\}.$
There exists $\varepsilon_{*}>0$ 
(dependent only on $\rho,$ $L_{1}$ and $L_{2}$ and not on $T$) 
such that if  $\varepsilon \in(0,\varepsilon_{*})$,
and   if
\begin{align}
 &(a_{1,0}(0))^{2}+(a_{2,0}(0))^{2}\leq M_{1,1}\varepsilon/4,\nonumber \\
 &|a_{2,0}(0)|\leq M_{2,1}\varepsilon/2, \nonumber \\
 &(a_{0,1}(0))^{2}+(a_{0,2}(0))^{2}\leq M_{1,2}\varepsilon/4, \nonumber \\
 &|a_{0,2}(0)|\leq M_{2,2}\varepsilon/2, \nonumber \\
 &\|w(0)\|_{\mathcal{B}_{\rho}^{1}}\leq M_{3}\varepsilon^{3/2}/6, \nonumber
 \end{align}
 then the  Kuramoto-Sivashinsky equation \eqref{KS} on the torus $\mathbb{T}^2=[0,L_1]\times [0,L_2]$ with initial data
\begin{multline}\nonumber
\psi(x,y,0)=\bar{\psi}_{0}+a_{1,0}(0)\cos\left(\frac{2\pi x}{L_{1}}\right)
+a_{2,0}(0)\cos\left(\frac{4\pi x}{L_{1}}\right)
\\
+a_{0,1}(0)\cos\left(\frac{2\pi y}{L_{2}}\right)+a_{0,2}(0)\cos\left(\frac{4\pi y}{L_{2}}\right)+w_{0}(x,y)
\end{multline}
has a mild solution  that is analytic in space on $[0,T]$.
\end{theorem}

\begin{proof}
By Proposition \ref{uniformEstimate}, the family $\{\phi^{n}\}_{n\in \mathbb{N}}$, where $\phi^n$ is given in \eqref{iteratedPhi}, is  a uniformly bounded family of functions analytic in space, with radius of
analyticity $\rho$ independent of $t$, and continuous and bounded in $t\in [0,\infty)$. Upon passing to a subsequence if necessary, not relabeled,  we  may thus  find a limit as $n\rightarrow\infty$ that is analytic in space by Montel's Theorem, continuous and bounded in time.
This is enough regularity to pass to the limit in the mild formulation of the evolution equations.
Thus the limit of the iterates, $\phi,$ exists and solves \eqref{eq:phiEq} on $[0,\infty)$.

We next turn to show the existence of $\psi$ solving the 2D KSE \eqref{KS} on $[0,T]$ for an arbitrary $0<T<\infty$.
We may write $\psi=\phi +\bar{\psi}$, where  $\bar{\psi}$ solves \eqref{meanEquation} with initial condition $\bar{\psi}_{0}.$ As noted in the introduction, the initial value problem for $\bar{\psi}$ can be solved on the time interval $[0,T]$ as long as $\phi\in L^{2}([0,T];\dot{H}^{1})$, which is the case given the regularity established above on $\phi$.
\end{proof}

Notice that we state this theorem on the interval $[0,T]$ for arbitrary $T,$ rather than using the interval $[0,\infty).$
This is because we do not know  that the mean, $\bar{\psi},$ 
has a well-defined limit as $t\to\infty$ since the mild formulation for $\phi$ does not imply that 
$\phi\in L^2([0,\infty);\dot{H}^1)$.  Nevertheless we have achieved the stated goal, showing that the 
two-dimensional Kuramoto-Sivashinsky equation has small solutions for all time in the presence of two linearly 
growing Fourier modes, one in each direction.

\begin{remark}
\begin{enumerate}

\item Our goal in this work is to establish global existence, therefore we do not treat in detail the uniqueness of the solution. However, uniqueness in $C([0,T],L^2)$ for any $0<T<\infty$ follows in a manner similar to that for viscous Hamilton-Jacobi equations and other semilinear parabolic equations (see for instance \cite[Proposition 1.1, page315]{TaylorPDEIII}).

\item As we have stated, our prior work \cite{ambroseMazzucato} proved global existence when
$L_{1}$ and $L_{2}$ are each in the interval $(0,2\pi),$ while here we have shown global existence when
they are each slightly larger than $2\pi.$  The present results do also extend to the case when only one of
these lengths is slightly larger than $2\pi,$ and the other is smaller. We also expect that arguments similar to those in this work and in \cite{ambroseMazzucato} will give the existence of a mild solution with initial data in $L^2$ under hypotheses akin to those in Theorem \ref{mainTheoremTechnical}. For brevity and clarity, we chose not to pursue this result here.

\item 
In \cite{ambroseMazzucato}, we proved two analyticity results for \eqref{KS} in the absence of linearly growing modes.
One of these results is that solutions of \eqref{KS} with data in $H^1$ are analytic for $t>0$ and the radius of analyticity can grow polynomially at first and then decay exponentially.   The other is that with small data in the Wiener algebra, 
the radius of analyticity grows at least linearly in time.  By contrast, in the present theorem, we have taken 
data in the spaces $B^m_\rho$ with $\rho>0;$ such data is analytic with radius of analyticity at least $\rho.$ The proof of 
Theorem \ref{mainTheoremTechnical} gives that the radius of analyticity of the remainder term 
remains analytic with radius of analyticity at least $\rho.$  
Regularity of local-in-time solutions for \eqref{KS} in Gevrey classes was studied in \cite{BiswasSwanson07}.
\end{enumerate}
\end{remark}

\section{Conclusion}\label{conclusion} 
We have shown the existence of small solutions of the
Kuramoto-Sivashinsky equation in two space dimensions for all time, when the size of the domain admits a linearly growing mode in each direction. To our knowledge, this is the first result of this kind.
The method of proof is new, in combining a dynamical systems approach for a finite number of modes with
function space estimates for the remaining infinitely many modes.  This approach raises the possibility that significant
further progress could be possible, extending beyond the present case of a pair of slightly growing modes,
by designing different Lyapunov functions or making different choices of function spaces.  
Another possible area of extension is extracting more detailed information about the solutions; while we showed that
$w\sim\varepsilon^{3/2},$ a finer description of amplitudes could be made for the modes encompassed by $w.$
The method of the present work could also be extended to other systems, including more fundamental systems
in flame propagation.  That is, the Kuramoto-Sivashinsky equation is a weakly nonlinear model, and can be 
proved to be a valid approximation for coordinate-free models; global existence of solutions for 
these coordinate-free and other models is of interest (see \cite{akersAmbroseFlame},
\cite{ambroseHadadifardWright}, \cite{frankelSivashinsky1987}).

\bibliography{2dksGrowingMode.bib}{}
\bibliographystyle{plain}

\end{document}